\documentclass[11pt]{article}

\usepackage{amsmath}
\usepackage{amsfonts}
\usepackage{amsthm}
\usepackage[latin1]{inputenc}
\usepackage[T1]{fontenc}
\usepackage{lmodern}
\usepackage{titleps}
\usepackage{color}

\date{}

\newtheorem{theo}{Theorem}
\newtheorem{lemma}{Lemma}

\newtheorem{coro}{Corollary}
\newtheorem{prop}{Proposition}

\begin{document}

\title{Differential operator for discrete Gegenbauer--Sobolev orthogonal  polynomials: eigenvalues and asymptotics}

\author{
Lance L. Littlejohn$^{1}$, Juan F. Ma\~{n}as--Ma\~{n}as$^{2}$,\\ Juan J. Moreno--Balc\'{a}zar$^{3}$, Richard Wellman$^{4}$\\ \\
\scriptsize{$^1$Department of Mathematics, Baylor University,  United States.}\\
\scriptsize{$^{2,3}$Departamento de Matem\'{a}ticas, Universidad de Almer\'{i}a, Spain.}\\
\scriptsize{$^3$Instituto Carlos I de F\'{i}sica Te\'{o}rica y Computacional, Spain.}\\
\scriptsize{$^4$Department of Mathematics, Westminster College SLC, United States.}\\
\scriptsize{{$^{1}$Lance\_Littlejohn@baylor.edu, $^{2}$jmm939@ual.es, $^{3}$balcazar@ual.es, $^{4}$rwellman@westminstercollege.edu}}
}

\maketitle

\begin{abstract}

We consider the following discrete Sobolev inner product involving the Gegenbauer weight
$$(f,g)_S:=\int_{-1}^1f(x)g(x)(1-x^2)^{\alpha}dx+M\big[f^{(j)}(-1)g^{(j)}(-1)+f^{(j)}(1)g^{(j)}(1)\big],$$
where $\alpha>-1,$ $j\in \mathbb{N}\cup \{0\},$ and $M>0.$ Let $\{Q_n^{(\alpha,M,j)}\}_{n\geq0}$ be the sequence of orthogonal  polynomials with respect to the above inner product. These polynomials are eigenfunctions of a differential operator $\mathbf{T}. $ We establish the asymptotic behavior of the corresponding eigenvalues. Furthermore, we calculate the exact value
$$r_0 = \lim_{n\rightarrow \infty}\frac{\log \left(\max_{x\in [-1,1]} |\widetilde{Q}_n^{(\alpha,M,j)}(x)|\right)}{\log \widetilde{\lambda}_n},$$
where $\{\widetilde{Q}_n^{(\alpha,M,j)}\}_{n\geq0}$ are  the sequence of orthonormal   polynomials with respect to this Sobolev inner product. This value $r_0$ is related to the convergence of a series in a  left--definite space. Finally, we study the Mehler--Heine type asymptotics for $\{Q_n^{(\alpha,M,j)}\}_{n\geq0}.$

\end{abstract}


\section{Introduction}

In the framework of Sobolev orthogonality, we consider the nonstandard inner product
\begin{equation}\label{pro}
(f,g)_S:=\int_{-1}^1f(x)g(x)(1-x^2)^{\alpha}dx+M\big[f^{(j)}(-1)g^{(j)}(-1)+f^{(j)}(1)g^{(j)}(1)\big],
\end{equation}
where $\alpha>-1,$ $j\in \mathbb{N}\cup \{0\},$ and $M>0.$ It is usually known as a Gegenbauer--Sobolev inner product because its absolutely continuous part involves the classical Gegenbauer weight.  We denote by $\{Q_n^{(\alpha,M,j)}\}_{n\geq0}$ the sequence of orthogonal polynomials with respect to (\ref{pro}). Along this paper we also use the  Sobolev orthonormal polynomials, denoted by $\{\widetilde{Q}_n^{(\alpha,M,j)}\}_{n\geq0}= \left\{ \frac{Q_n^{(\alpha,M,j)}}{\sqrt{(Q_n^{(\alpha,M,j)}, Q_n^{(\alpha,M,j)})_S}}\right\}_{n\geq0}. $

Consider a sequence of polynomials, $\{P_n\}_{n\geq0},$ orthogonal with respect to a symmetric inner-product $\phi$ that also satisfy a (possibly infinite order) spectral differential equation. In \cite{Kwon} the authors give conditions for  polynomials orthogonal with respect to a related discrete Sobolev inner product of the form $(f,g)_S =\phi(f,g) + Mf^{(j)}(c) g^{(j)}(c)$ to  also satisfy a (possibly infinite order) spectral differential equation. H. Bavinck in \cite{Bavinck}  extended this result to polynomials $\{Q_n\}_{n\geq0}$ orthogonal with respect to discrete Sobolev inner products of the form  $(f,g)_S =\phi(f,g) + Mf^{(j)}(c_1)g^{(j)}(c_1) + Nf^{(k)}(c_2)g^{(k)}(c_2)$. Specifically Bavinck constructs a differential equation $\ell$ and eigenvalues  $\widetilde{\lambda}_n$ such that
\[ \ell [Q_n](x) = \sum_{i=1}^{\infty}a_i(x)Q_n^{(i)}(x) = \widetilde{\lambda}_n Q_n(x). \]
Central to the construction of this differential equation  is the reproducing polynomial kernel $K_n(x,y) := \sum_{i=0}^n\frac{P_i(x)P_i(y)}{\phi(P_i,P_i)}$.

Here we consider the Gegenbauer--Sobolev  polynomials $Q_n^{(\alpha,M,j)}$ orthogonal with respect to the discrete Sobolev--type inner product (\ref{pro}). We are interested in properties of the related polynomial kernel scaled by the eigenvalues:
\[K(x,y;r)=  \sum_{i=0}^{\infty}\widetilde{\lambda}_i^{-r} \frac{Q_i^{(\alpha,M,j)}(x)Q_i^{(\alpha,M,j)}(y)}{(Q_i^{(\alpha,M,j)},Q_i^{(\alpha,M,j)})_S}\]

 They are useful for some applications that two of the authors \cite{wellman}  is developing. Since
that work is unfinished we only give a briefly motivation without details. The results
on this paper are self--contained and do not depend on this motivation.

Let $\left(H,\right<\cdot,\cdot\left>_H\right)$ be the completion space of polynomials under the inner product $(f,g)_S$. Let $\mathbf{T}$ be the self-adjoint operator in $H$ generated by the differential expression $\mathbf{L} + M\mathbf{A}$, where $\mathbf{L}$ is the linear differential operator associated with the Gegenbauer polynomials and $\mathbf{A}$ is an operator that we will define later. $\mathbf{T}$ exists as an unbounded operator in $H$ since $\widetilde{\lambda}_n\rightarrow\infty$. In this case the \emph{left--definite} space $H_r(\mathbf{T})$ with inner--product
\[\left<f,g\right>_{r}:=\left<\mathbf{T}^rf,g\right>_H\]
 on the linear manifold $\mathcal{D}(\mathbf{T}^{r/2})$ yields a Hilbert space. Furthermore we may take the power $\mathbf{T}^s$ as a self-adjoint operator in a left-definite space (see \cite{littlejohn-wellman} for details).

Now take $r_0$ to be the least number such that for each $r>r_0$ the  kernel $K(x,y;r)$ converges both absolutely and in the left--definite space $H_{r}(\mathbf{T}).$ Then for $r>r_0$ the sequence $\left\{\frac{Q_n^{(\alpha,M,j)}(x)}{\sqrt{\widetilde{\lambda}_n^{r}(Q_n^{(\alpha,M,j)},Q_n^{(\alpha,M,j)})_S}}\right\}_{n\ge 0}$ forms a complete, orthonomal basis for $H_{r}(\mathbf{T})$ and the reproducing property  follows from the Parseval identity:
 \[\left<\mathbf{T}^r K(x,\cdot;r_0^++r),f\right>_{H_{r_0^+}}=\left<K(x,\cdot;r_0^+),f\right>_{H_{r_0^+}} = f(x).\]

 Notice this gives $K(x,y;r)$ as the reproducing kernel for the $r-r_0^+$ left--definite space of the left--definite operator acting in the $r_0^+$ left--definite space generated by $\mathbf{T}$.

It can be shown \cite{wellman} that
\begin{equation*}
r_0 = \lim_{n\rightarrow \infty}\frac{\log \left(\sup_{x\in [-1,1]} |\widetilde{Q}_n^{(\alpha,M,j)}(x)|\right)}{\log \widetilde{\lambda}_n}  \end{equation*}
 In this paper we find a value for $r_0$ valid for all $\alpha > -1$, $M>0$ and $j\in \mathbb{N}\cup \{0\}$.

Another goal of this paper will be to obtain a Mehler--Heine formula for the polynomials $Q_n^{(\alpha,M,j)}$. These Mehler--Heine type formulae are interesting twofold: they provide the scaled asymptotics for $Q_n^{(\alpha,M,j)}$ on compact sets of the complex plane and they supply us with asymptotic information about the location of the zeros of these polynomials in terms of the zeros of other known special functions.

The structure of the paper is the following: in Section 2 we give a background about Gegenbauer orthogonal polynomials, $C_n^{(\alpha)}, $ introducing the properties of these polynomials that will be used along the paper. In Section 3 we establish connection formulae between the polynomials $C_n^{(\alpha)} $ and $Q_n^{(\alpha,M,j)}$ which are useful to state an upper bound for $||Q_n^{(\alpha,M,j)}||_{\infty}=\max_{x \in [-1,1]}|Q_n^{(\alpha,M,j)}(x) |.$ In Section 4 we give the asymptotic behavior of the eigenvalues associated with the linear differential operator $\mathbf{T}=\mathbf{L}+M \mathbf{A}$, i.e. $$ \mathbf{T}Q_n^{(\alpha,M, j)}(x)=(\mathbf{L}+M\mathbf{A})Q_n^{(\alpha,M, j)}(x)=\widetilde{\lambda}_nQ_n^{(\alpha,M,j)}(x). $$
Thus, we calculate the value of $r_0$ in Theorem \ref{main1}. Finally, Section 5  is devoted to the study of the Mehler--Heine type asymptotics for the Sobolev polynomials $Q_n^{(\alpha,M,j)}$ giving additional information about the asymptotic behavior of the zeros of  these polynomials.

Along the text we will use the following notation: if $a_n$ and $b_n$ are two sequences of real numbers, then  $a_n\approx b_n$ means  $\displaystyle\lim_{n\to+\infty}a_n/b_n=1.$

\section{Gegenbauer orthogonal polynomials: a background}

In \cite{sz}  Gegenbauer polynomials are considered as those polynomials which are orthogonal with respect to the inner product
$$
(f,g)=\int_{-1}^1f(x)g(x)(1-x^2)^{\lambda-1/2}dx, \quad \lambda>-1/2.  $$
They are denoted by $P_n^{(\lambda)}$ and it is used the normalization
$$
P_n^{(\lambda)}(1)= \frac{\Gamma(n+2\lambda)}{\Gamma(n+1) \Gamma(2\lambda)}.
$$
We take $\alpha:=\lambda-1/2$ and consider the polynomials $\displaystyle{\frac{P_n^{(\lambda)}(x)}{P_n^{(\lambda)}(1)}}. $ We denote by $\{C_n^{(\alpha)}\}_{n\geq0}$ this sequence of Gegenbauer orthogonal polynomials (this normalization is also used in \cite{AFB1999}). In this way, it is obvious that the polynomials $C_n^{(\alpha)}$ are orthogonal with respect to
$$(f,g)_{\alpha}:=\int_{-1}^1f(x)g(x)(1-x^2)^{\alpha}dx,$$
 with  $C_n^{(\alpha)}(1)=1,$ and using the symmetry of these polynomials we have $C_n^{(\alpha)}(-1)=(-1)^n.$ So, the inner product  (\ref{pro}) can be rewritten as
$$(f,g)_S=(f,g)_{\alpha}+M\big[f^{(j)}(-1)g^{(j)}(-1)+f^{(j)}(1)g^{(j)}(1)\big].$$

Now, we  recall some properties of  Gegenbauer orthogonal  polynomials. These properties can be found in \cite{AFB1999} or \cite{sz} among others.
\begin{enumerate}
  \item \textbf{Derivatives:} \begin{equation}\label{deri}  \big(C_n^{(\alpha)}(x)\big)^{(k)}:=\frac{\partial^k C_n^{(\alpha)}(x)}{\partial x^k}=\frac{(-1)^k(n+2\alpha+1)_k(-n)_k}{2^k(\alpha+1)_k}C_{n-k}^{(\alpha+k)}(x), \end{equation} with $k=0,1,\dots,$ where $(a)_k$ denotes the Pochhammer's symbol,  i.e.,  $(a)_{k}= a (a+ 1)\cdots (a+k-1) = \frac{\Gamma(a+k)}{\Gamma(a)},\quad k\geq 1, \quad (a)_{0}=1.$
  \item \textbf{Differential equation and eigenvalues:} \begin{equation*}\label{DE} (x^2-1)\big(C_n^{(\alpha)}(x)\big)^{(2)}+2(\alpha+1)x \big(C_n^{(\alpha)}(x)\big)^{(1)}=\lambda_n C_n^{(\alpha)}(x), \end{equation*}
      \begin{equation}\label{EiV}\lambda_n=n(n+2\alpha+1).\end{equation}
  \item \textbf{Leading coefficient:} \begin{equation}\label{lc} k_n(\alpha):=\frac{(n+2\alpha+1)_n}{2^n(\alpha+1)_n}=
      \frac{\Gamma(2n+2\alpha+1)\Gamma(\alpha+1)}{2^n\Gamma(n+\alpha+1)\Gamma(n+2\alpha+1)}. \end{equation}
  \item \textbf{Squared norm:} \begin{equation}\label{norm} ||C_n^{(\alpha)}||^2_{\alpha}:=\int_{-1}^1(C_n^{(\alpha)}(x))^2(1-x^2)^{\alpha} dx = \frac{2^{2\alpha+1}\Gamma^2(\alpha+1)\Gamma(n+1)}{(2n+2\alpha+1)\Gamma(n+2\alpha+1)}. \end{equation}
\end{enumerate}

Now, we are going to calculate some limits which  will be used later. To do this, we take into account (see, for example, \cite[f. (5.11.13)]{askey} or \cite[f. (7)]{mavjs})

\begin{equation}\label{stirling}
\lim_{n\to+\infty}\frac{n^{b-a}\Gamma(n+a)}{\Gamma(n+b)}=1.
\end{equation}
Thus,
\begin{equation}\label{pochamer}
\lim_{n\to+\infty}\frac{(n+a)_k}{n^k}=\lim_{n\to+\infty}\frac{\Gamma(n+a+k)}{\Gamma(n+a)n^k}=1.
\end{equation}

Furthermore,
\begin{equation}\label{pochamer-1}
\lim_{n\to+\infty}\frac{(-n+a)_k}{n^k}=(-1)^k.
\end{equation}

In the next lemma we provide some useful asymptotic behaviors of Gegenbauer polynomials.

\begin{lemma}\label{prop cnk(1)} For $k\in \mathbb{N}\cup \{0\}$, we have
\begin{equation}
\label{cnk(1)}
\lim_{n\to+\infty}\frac{\left(C_n^{(\alpha)}(1)\right)^{(k)}}{n^{2k}}=\frac{1}{2^k(\alpha+1)_k}.
\end{equation}
Furthermore,
\begin{equation}\label{asNORM}
\lim_{n\to+\infty}||C_n^{(\alpha)}||^2_{\alpha}n^{2\alpha+1}=
2^{2\alpha}\Gamma^2(\alpha+1).
\end{equation}

\end{lemma}
\begin{proof} Using (\ref{deri}), (\ref{pochamer}-\ref{pochamer-1}) and  the fact that $C_n^{\alpha}(1)=1, $ we obtain
\begin{eqnarray*}
\lim_{n\to+\infty}\frac{\left(C_n^{(\alpha)}(1)\right)^{(k)}}{n^{2k}}&=& \frac{(-1)^k}{2^k(\alpha+1)_k}\lim_{n\to\infty}\frac{(n+2\alpha+1)_k(-n)_k}
{ n^{2k}}\\
&=&\frac{(-1)^k}{2^k(\alpha+1)_k}(-1)^k=\frac{1}{2^k(\alpha+1)_k}.\\
\end{eqnarray*}
Formula (\ref{asNORM}) is deduced in a straightforward way from (\ref{norm}) using (\ref{stirling}).
\end{proof}

We will use the following notation:
\begin{eqnarray}
K_n^{(j,k)}(x,y)&=&\sum_{i=0}^{n}\frac{\left(C_i^{(\alpha)}(x)\right)^{(j)}\left(C_i^{(\alpha)}(y)\right)^{(k)}}
{||C_i^{(\alpha)}||^2_{\alpha}}, \nonumber\\
\kappa_{2n}^{(j,k)}(x,y)&=&\sum_{i=0}^{n}\frac{\left(C_{2i}^{(\alpha)}(x)\right)^{(j)}\left(C_{2i}^{(\alpha)}(y)\right)^{(k)}}
{||C_{2i}^{(\alpha)}||^2_{\alpha}}, \label{evenkernels}\\
\widetilde{\kappa}_{2n}^{(j,k)}(x,y)&=&\sum_{i=0}^{n}\frac{\left(C_{2i+1}^{(\alpha)}(x)\right)^{(j)}\left(C_{2i+1}^{(\alpha)}(y)\right)^{(k)}}
{||C_{2i+1}^{(\alpha)}||^2_{\alpha}} \label{oddkernels} .
\end{eqnarray}
Notice that $K_n^{(0,0)}(x,y)= K_n(x,y)$ are the usual kernel polynomials associated with Gegenbauer polynomials.

\begin{prop}\label{kernel(k,s)(1,1)}
 Let $k$ and $s$ be nonnegative integer numbers. Then,
\begin{eqnarray}\label{kernel(1,1)}
\lim_{n\to+\infty} \frac{K_{n-1}^{(k,s)}(1,1)}{n^{2k+2s+2\alpha+2}}&=&
\frac{1}{2^{2\alpha+k+s+1}} C_{k,s}, \\
\label{kappa(k,s)(1,1)}
\lim_{n\to+\infty}\frac{\kappa_{2(n-1)}^{(k,s)}(1,1)}{n^{2k+2s+2\alpha+2}}&=& \lim_{n\to+\infty}\frac{\widetilde{\kappa}_{2(n-1)}^{(k,s)}(1,1)}{n^{2k+2s+2\alpha+2}}=2^{k+s}C_{k,s},
\end{eqnarray}
where $$C_{k,s}= \frac{1}{(k+s+\alpha+1)\Gamma(\alpha+k+1)\Gamma(\alpha+s+1)}.$$
\end{prop}

\begin{proof}  First,  we observe that
$$n^{2\alpha+2k+2s+2}-(n-1)^{2\alpha+2k+2s+2}\approx (2\alpha+2k+2s+2)n^{2\alpha+2k+2s+1}.$$

\noindent Using  Stolz's criterion,  (\ref{cnk(1)}), and (\ref{asNORM}) we get
\begin{eqnarray*}
\lim_{n\to+\infty} \frac{K_{n-1}^{(k,s)}(1,1)}{n^{2k+2s+2\alpha+2}}
&=&\lim_{n\to+\infty}\frac{K_{n-1}^{(k,s)}(1,1)-K_{n-2}^{(k,s)}(1,1)}{n^{2k+2s+2\alpha+2}-(n-1)^{2k+2s+2\alpha+2}}\\
&=&\lim_{n\to+\infty}\frac{\frac{\left(C_{n-1}^{(\alpha)}(1)\right)^{(k)}\left(C_{n-1}^{(\alpha)}(1)\right)^{(s)}}{||C_{n-1}^{(\alpha)}||^2_{\alpha}}}
{2(\alpha+k+s+1)n^{2\alpha+2k+2s+1}}\\
&=&\frac{1}{2(\alpha+k+s+1)}\lim_{n\to+\infty}\frac{\left(C_{n-1}^{(\alpha)}(1)\right)^{(k)}}{n^{2k}}
\frac{\left(C_{n-1}^{(\alpha)}(1)\right)^{(s)}}{n^{2s}}\frac{1}{||C_{n-1}^{(\alpha)}||^2_{\alpha}n^{2\alpha+1}}\\
&=&\frac{1}{2(\alpha+k+s+1)}\frac{1}{2^k(\alpha+1)_k}\frac{1}{2^s(\alpha+1)_s}\frac{1}{2^{2\alpha}\Gamma^2(\alpha+1)}
\end{eqnarray*}

Finally, using $\displaystyle (\alpha+1)_k=\frac{\Gamma(\alpha+k+1)}{\Gamma(\alpha+1)}$ we obtain (\ref{kernel(1,1)}).
To establish (\ref{kappa(k,s)(1,1)}) we can proceed in the same way.

\end{proof}

\section{Connection formulae and some asymptotic behaviors}

 It is well known that $\{C_i^{(\alpha)}\}_{i=0}^m$ constitute a basis of the linear space $\mathbb{P}_m[x]$ of polynomials with real coefficients and degree at most $m$. Therefore, the Gegenbauer--Sobolev  polynomials orthogonal with respect to (\ref{pro}), with leading coefficient $k_n(\alpha)$ given in (\ref{lc}), can be expressed as
\begin{eqnarray*}\label{expansion}
Q_{2n}^{(\alpha,M,j)}(x)&=&C_{2n}^{(\alpha)}(x)+\sum_{i=0}^{n-1}a_{2n,2i}C_{2i}^{(\alpha)}(x), \\ \label{expansion2}
Q_{2n+1}^{(\alpha,M,j)}(x)&=&C_{2n+1}^{(\alpha)}(x)+\sum_{i=0}^{n-1}a_{2n+1,2i+1}C_{2i+1}^{(\alpha)}(x).
\end{eqnarray*}
\noindent Thus, applying a well--established procedure (see, for example, \cite[Sect. 2]{maRon} among others),  we can deduce the following connection formulae.

\begin{prop} \label{firstconn} We have,
\begin{eqnarray}\label{GSeven}
Q_{2n}^{(\alpha,M,j)}(x)&=&C_{2n}^{(\alpha)}(x)-\frac{2M\big(C_{2n}^{(\alpha)}(1)\big)^{(j)}\kappa_{2(n-1)}^{(j,0)}(1,x)}{1+2M\kappa_{2(n-1)}^{(j,j)}(1,1)}
\\
\label{GSodd}
 Q_{2n+1}^{(\alpha,M, j)}(x)&=&
C_{2n+1}^{(\alpha)}(x)-\frac{2M\big(C_{2n+1}^{(\alpha)}(1)\big)^{(j)}\widetilde{\kappa}_{2(n-1)}^{(j,0)}(1,x)}{1+2M\widetilde{\kappa}_{2(n-1)}^{(j,j)}(1,1)}
\end{eqnarray}
\end{prop}
\begin{proof} For $i=0,\ldots, n-1$ fixed, we have
\begin{eqnarray*}
0&=&(Q_{2n}^{(\alpha,M,j)}(x),C_{2i}^{(\alpha)}(x))_{S}
=\left(C_{2n}^{(\alpha)}(x)+\sum_{k=1}^{n-1}a_{2n,2k}C_{2k}^{(\alpha)}(x),C_{2i}^{(\alpha)}(x)\right)_{S}   \\
&=&\left(C_{2n}^{(\alpha)}(x),C_{2i}(x)^{(\alpha)}\right)_{\alpha}+\sum_{k=0}^{n-1}a_{2n,2k}
\left(C_{2k}^{(\alpha)}(x),C_{2i}(x)^{(\alpha)}\right)_{\alpha}\\
&+&M\big[\big(Q_{2n}^{(\alpha,M,j)}(-1)\big)^{(j)}\big(C_{2i}^{(\alpha)}(-1)\big)^{(j)}+
\big(Q_{2n}^{(\alpha,M,j)}(1)\big)^{(j)}\big(C_{2i}^{(\alpha)}(1)\big)^{(j)}\big]\\
&=&a_{2n,2i}||C_{2i}^{(\alpha)}||^2_{\alpha}+2M\big(Q_{2n}^{(\alpha,M,j)}(1)\big)^{(j)}\big(C_{2i}^{(\alpha)}(1)\big)^{(j)},
\end{eqnarray*}
thus,  $$a_{2n,2i}=\frac{-2M\big(Q_{2n}^{(\alpha,M,j)}(1)\big)^{(j)}\big(C_{2i}^{(\alpha)}(1)\big)^{(j)}}{||C_{2i}^{(\alpha)}||^2_{\alpha}},$$ and
\begin{eqnarray*}
Q_{2n}^{(\alpha,M,j)}(x)&=&C_{2n}^{(\alpha)}(x)+\sum_{i=0}^{n-1}\frac{-2M
\big(Q_{2n}^{(\alpha,M,j)}(1)\big)^{(j)}\big(C_{2i}^{(\alpha)}(1)\big)^{(j)}}{||C_{2i}^{(\alpha)}||^2_{\alpha}}C_{2i}^{(\alpha)}(x)\\
&=&C_{2n}^{(\alpha)}(x)-2M\big(Q_{2n}^{(\alpha,M,j)}(1)\big)^{(j)}\sum_{i=0}^{n-1}\frac{\big(C_{2i}^{(\alpha)}(1)\big)^{(j)}C_{2i}^{(\alpha)}(x)}
{||C_{2i}^{(\alpha)}||^2_{\alpha}}\\
&=&C_{2n}^{(\alpha)}(x)-2M\big(Q_{2n}^{(\alpha,M,j)}(1)\big)^{(j)}\kappa_{2(n-1)}^{(j,0)}(1,x).
\end{eqnarray*}
Then, we derive the above expression $j$ times  and  evaluate at  $x=1$ obtaining
\begin{equation} \label{jderb1}
\big(Q_{2n}^{(\alpha,M,j)}(1)\big)^{(j)}=\frac{\big(C_{2n}^{(\alpha)}(1)\big)^{(j)}}{1+2M\kappa_{2(n-1)}^{(j,j)}(1,1)},
\end{equation}
which proves (\ref{GSeven}). For the odd case, relation (\ref{GSodd})  is established in the same way.
\end{proof}

Proposition \ref{firstconn} is very useful to obtain the following relative asymptotics at the point $x=1.$
\begin{prop} \label{relasy1}
Let  $k$ be a nonnegative integer. Then, we have
\begin{equation*}\label{ASeven}
\lim_{n\to+\infty}\frac{\big(Q_{n}^{(\alpha,M,j)}(1)\big)^{(k)}}{\big(C_{n}^{(\alpha)}(1)\big)^{(k)}}=
\frac{k-j}{j+k+\alpha+1}.
\end{equation*}
\end{prop}
\begin{proof}
We only prove the even case since the proof for the odd case is the same one. We derive the expression (\ref{GSeven}) $k$ times and evaluate at $x=1$. Then, we divide by $\displaystyle \big(C_{2n}^{(\alpha)}(1)\big)^{(k)}$ and use the limit relations (\ref{cnk(1)}) and (\ref{kappa(k,s)(1,1)}).
\begin{eqnarray*}
\lim_{n\to+\infty}\frac{\big(Q_{2n}^{(\alpha,M,j)}(1)\big)^{(k)}}{\big(C_{2n}^{(\alpha)}(1)\big)^{(k)}}
&=&1-\lim_{n\to+\infty}\frac{2M\big(C_{2n}^{(\alpha)}(1)\big)^{(j)}\kappa_{2(n-1)}^{(j,k)}(1,1)}
{\left(1+2M\kappa_{2(n-1)}^{(j,j)}(1,1)\right)\big(C_{2n}^{(\alpha)}(1)\big)^{(k)}}\\
&=&1-\lim_{n\to+\infty}\frac{2M\frac{\big(C_{2n}^{(\alpha)}(1)\big)^{(j)}}{n^{2j}}\frac{\kappa_{2(n-1)}^{(j,k)}(1,1)}{n^{2j+2k+2\alpha+2}}}
{\frac{\big(C_{2n}^{(\alpha)}(1)\big)^{(k)}}
{n^{4j+2k+2\alpha+2}}+2M\frac{\kappa_{2(n-1)}^{(j,j)}(1,1)}{n^{4j+2\alpha+2}}\frac{\big(C_{2n}^{(\alpha)}(1)\big)^{(k)}}{n^{2k}}}\\
&=&1-\frac{(2j+\alpha+1)(\alpha+1)_k\Gamma(\alpha+j+1)}{(\alpha+1)_j(j+k+\alpha+1)\Gamma(\alpha+k+1)}\\
&=&1-\frac{2j+\alpha+1}{j+k+\alpha+1}=\frac{k-j}{j+k+\alpha+1}.
\end{eqnarray*}
\end{proof}

Now, we are going to establish that the norm of the Gegenbauer--Sobolev orthogonal polynomials, induced by the nonstandard inner product (\ref{pro}),  behaves like the norm of classical Gegenbauer polynomials.
\begin{prop} \label{p-ABnorm2n}
We have,
\begin{equation*}\label{ABnorm2n}
\lim_{n\to \infty} \frac{||Q_{n}^{(\alpha,M,j)}||_S}{||C_{n}^{(\alpha)}||_{\alpha}}= 1.
\end{equation*}
\end{prop}
\begin{proof} Again, we only prove the even case.
\begin{eqnarray*}
\big(Q_{2n}^{(\alpha,M,j)},Q_{2n}^{(\alpha,M,j)}\big)_S
=\big(Q_{2n}^{(\alpha,M,j)},C_{2n}^{(\alpha)}\big)_S
=||C_{2n}^{(\alpha)}||^2_{\alpha}+2M\big(Q_{2n}^{(\alpha,M,j)}(1)\big)^{(j)}\big(C_{2n}^{(\alpha)}(1)\big)^{(j)}.
\end{eqnarray*}
It is enough to observe that applying  (\ref{cnk(1)}),  (\ref{kappa(k,s)(1,1)}), and (\ref{jderb1}),  we get
\begin{eqnarray*}
\lim_{n\to+\infty}2M\big(Q_{2n}^{(\alpha,M,j)}(1)\big)^{(j)}\big(C_{2n}^{(\alpha)}(1)\big)^{(j)}=
\lim_{n\to+\infty}\frac{2M \left(\big(C_{2n}^{(\alpha)}(1)\big)^{(j)}\right)^2}{1+2M\kappa_{2(n-1)}^{(j,j)}(1,1)}=0,
\end{eqnarray*}
which proves the result.
\end{proof}

The number of terms of the connection formula given in Proposition \ref{firstconn} depends on $n,$ so this number increases when $n$ grows. To avoid this, we can give another \textit{connection} formula in which the polynomials $Q_n^{(\alpha,M,j)}$ can be expressed as a finite linear combination of polynomials not depending on $n$.
\begin{prop}\label{th-rcf}
There exists a family of real numbers $\{\gamma_{n,i}\}_{i=0}^{j+1}$, not identically zero, such that the following connection formula
holds
\begin{equation}\label{RCF}
Q_n^{(\alpha,M,j)}(x)=\sum_{i=0}^{j+1}\gamma_{n,i}(1-x^2)^i\big(C_{n-i}^{(\alpha+i)}(x)\big)^{(i)},\qquad  n\geq 2j+2.
\end{equation}
\end{prop}
\begin{proof} It was established in a similar framework in \cite[Th. 1]{dls2015}, although in that paper the discrete part of the Sobolev inner product is located at only one point $c$. Being the same procedure,  we prefer to omit the details.
\end{proof}

\begin{prop}\label{conv gamma i}
Let $\{\gamma_{n,i}\}_{i=0}^{j+1}$ be the coefficients given in (\ref{RCF}). Then, $$\lim_{n\to+\infty}\gamma_{n,i}=\gamma_i\in \mathbb{R},\qquad 0\leq i \leq j+1.$$
\end{prop}

\begin{proof} As we have commented in the previous proposition, this result was also established in a similar context in \cite[Th. 1]{dls2015} (see also \cite[Th. 2]{mavjs}). But, now the discrete part of our  inner product (\ref{pro}) is concentrated in two points, not only in one like the references cited. Anyway, the technique is the same.  However, we include the main lines of the proof because in this concrete case  we can establish the exact value of $\gamma_i$  which has interest by itself. Thus, in the first step we derive formula (\ref{RCF}) $k$ times and evaluate at $x=1$. Thus, for  $0 \leq k \leq j+1, $  we get
\begin{equation}\label{fpag9}
\big(Q_{n}^{(\alpha,M,j)}(1)\big)^{(k)}
=\sum_{i=0}^k \gamma_{n,i} \binom{k}{i}(-1)^i i!
\left(\sum_{l=0}^{k-i} \frac{i!}{(i-l)!}2^{(i-l)}\big(C_{n-i}^{(\alpha+i)}(1)\big)^{(k-l)}\right).
\end{equation}
Now, we divide (\ref{fpag9}) by $\big(C_{n}^{(\alpha)}(1)\big)^{(k)}$ and taking into account Proposition \ref{relasy1} we can deduce the result if and only if \newline $\lim_{n\to+\infty}\frac{\big(C_{n-i}^{(\alpha+i)}(1)\big)^{(k-l)}}
{\big(C_{n}^{(\alpha)}(1)\big)^{(k)}}\in \mathbb{R}$, with $0\leq l \leq k-i.$ But this is true by Lemma \ref{prop cnk(1)}. In fact,

\begin{equation}\label{fpag92}
\lim_{n\to+\infty}\frac{\big(C_{n-i}^{(\alpha+i)}(1)\big)^{(k-l)}}{\big(C_{n}^{(\alpha)}(1)\big)^{(k)}}
=\left\{
     \begin{array}{ll}
       \frac{(\alpha+1)_k}{(\alpha+i+1)_k}, & \hbox{if $\quad l=0$;} \\
       0, & \hbox{if $\quad 1\leq l\leq k-i$.}
     \end{array}
   \right.
\end{equation}
\end{proof}

We have proved that the sequences $\{\gamma_{n,i}\}_n$ are convergent with $i\in\{0,\ldots, j+1\}$ when $n\to \infty.$ Now, we want to compute explicitly the corresponding limits $\gamma_i$ with $0 \leq i \leq j + 1.$ For $i = 0$, using Proposition \ref{relasy1}, we get
$$ \lim_{n\to+\infty}\frac{Q_{n}^{(\alpha,M,j)}(1)}{C_{n}^{(\alpha)}(1)}=\lim_{n\to+\infty}Q_{n}^{(\alpha,M,j)}(1)=
\lim_{n\to+\infty}\gamma_{n,0}=\frac{-j}{j+\alpha+1}.$$
Thus, we can construct a recursive algorithm based on (\ref{fpag9}) and, paying attention to (\ref{fpag92}), we deduce easily the next result.

\begin{coro}\label{limit gamma i}
\begin{equation}\label{gamma_i}
\gamma_i=\left\{
           \begin{array}{ll}
            \displaystyle \frac{-j}{j+\alpha+1}, & \hbox{if $\quad i=0$;} \\
            \displaystyle (-1)^i\frac{\frac{i-j}{j+i+\alpha+1}-\sum_{k=0}^{i-1}\gamma_{k}\binom{i}{k}(-2)^k k!\frac{(\alpha+1)_i}{(\alpha+k+1)_i}} {\frac{2^i(\alpha+1)_i}{(\alpha+i+1)_i}}, & \hbox{if $\quad 1\leq i\leq j+1$.}
           \end{array}
         \right.
\end{equation}
\end{coro}

\medskip

Finally, we give an upper bound of the uniform norm of the Sobolev polynomials. This result will be useful to establish one of our main  target in Section 4.

\begin{theo}\label{cota Bn}
Let $\displaystyle Q_n^{(\alpha,M,j)}(x)$ be the orthogonal polynomials with respect to (\ref{pro}), then
\begin{equation*}
||Q_n^{(\alpha,M,j)}||_{\infty}:=\max_{x\in [-1,1]}\left|Q_n^{(\alpha,M,j)}(x)\right|\leq \left\{
                                                    \begin{array}{ll}
                                                      \frac{3j+2\alpha+2}{j+\alpha+1}+D, & \hbox{if $\quad\alpha\geq-1/2$;} \\
                                                      F\ n^{-\alpha-1/2}, & \hbox{if $\quad-1<\alpha<-1/2$,}
                                                    \end{array}
                                                  \right.
\end{equation*}
when $n\to+\infty,$ being $D$ and $F$ positive constants independent of $n$.
\end{theo}

\begin{proof}  Taking $\alpha=\lambda-1/2$ and considering the expression  (4.7.1) in \cite{sz},   we have
\begin{equation}\label{RGJ}
C_n^{(\alpha)}(x)=\frac{P_n^{(\lambda)}(x)}{P_n^{(\lambda)}(1)}=\frac{\Gamma(n+1)\Gamma(2\alpha+1)}{\Gamma(n+2\alpha+1)}P_n^{(\lambda)}(x)=
\frac{\Gamma(n+1)\Gamma(\alpha+1)}{\Gamma(n+\alpha+1)} P_n^{(\alpha, \alpha)}(x),
\end{equation}
where $P_n^{(\alpha, \beta)}$ are the classical Jacobi polynomials orthogonal with respect to the weight function $(1-x)^{\alpha}(1+x)^{\beta}, $
$\alpha, \beta>-1.$

Now, we use a uniform bound of $|P_n^{(\alpha,\alpha)}|$ given in  \cite[f. (22.14.1)]{abra}, i.e., for $-1\leq x\leq 1$

\begin{equation} \label{bound-jacobi} |P_n^{(\alpha,\alpha)}(x)| \leq \left\{
                                \begin{array}{ll}
                                  \displaystyle  P_n^{(\alpha,\alpha)}(1)=\binom{n+\alpha}{n}\approx n^{\alpha}, & \hbox{if $\quad \alpha\geq-1/2$;} \\
                                  \displaystyle |P_n^{(\alpha,\alpha)}(0)|\approx n^{-1/2}, & \hbox{if $\quad -1<\alpha<-1/2$.}
                                \end{array}
                              \right.
\end{equation}
To prove the result we use different approaches according to each case.
\begin{itemize}
  \item Case $\alpha\geq-1/2.$ From  (\ref{RGJ}) and (\ref{bound-jacobi}), it is clear that $\max_{x\in[-1,1]}|C_n^{(\alpha)}(x)|=1, $ and this maximum is reached at $x=1.$ We only prove the even case since the proof of the odd case is totally similar. First, we have
\begin{eqnarray*}
\max_{x\in[-1,1]}|\kappa_{2(n-1)}^{(j,0)}(1,x)|
&=&\max_{x\in[-1,1]}\left|\sum_{i=0}^{n-1}\frac{\left(C_{2i}^{(\alpha)}(1)\right)^{(j)}C_{2i}^{(\alpha)}(x)}
{||C_{2i}^{(\alpha)}||^2_{\alpha}}\right|\\
&\leq&\sum_{i=0}^{n-1}\frac{\left(C_{2i}^{(\alpha)}(1)\right)^{(j)}\max_{x\in[-1,1]}\left|C_{2i}^{(\alpha)}(x)\right|}
{||C_{2i}^{(\alpha)}||^2_{\alpha}} \\
&=&\sum_{i=0}^{n-1}\frac{\left(C_{2i}^{(\alpha)}(1)\right)^{(j)}C_{2i}^{(\alpha)}(1)}
{||C_{2i}^{(\alpha)}||^2_{\alpha}}=\kappa_{2(n-1)}^{(j,0)}(1,1).
\end{eqnarray*}
Therefore, using (\ref{GSeven}) and the previous bound we get,
\begin{eqnarray*}
\max_{x\in[-1,1]}|Q_{2n}^{(\alpha,M,j)}(x)|&\leq& \max_{x\in[-1,1]}\left|C_{2n}^{(\alpha)}(x)\right|+
\max_{x\in[-1,1]}\left|\frac{2M\big(C_{2n}^{(\alpha)}(1)\big)^{(j)}\kappa_{2(n-1)}^{(j,0)}(1,x)}{1+2M\kappa_{2(n-1)}^{(j,j)}(1,1)}\right|\\
&=&1+\frac{2M\big(C_{2n}^{(\alpha)}(1)\big)^{(j)}\kappa_{2(n-1)}^{(j,0)}(1,1)}{1+2M\kappa_{2(n-1)}^{(j,j)}(1,1)}.
\end{eqnarray*}
On the other hand, we can observe that in the proof of Proposition \ref{relasy1} it was established for $k=0$ that
$$
\lim_{n\to \infty} \frac{2M\big(C_{2n}^{(\alpha)}(1)\big)^{(j)}\kappa_{2(n-1)}^{(j,0)}(1,1)}{1+2M\kappa_{2(n-1)}^{(j,j)}(1,1)}=\frac{2j+\alpha+1}{j+\alpha+1}.
$$
Thus, we claim that for $n$ large enough there exists a positive constant $D$ such that
\begin{equation*}
\max_{x\in[-1,1]}|Q_{2n}^{(\alpha,M,j)}(x)|\leq  1+\frac{2j+\alpha+1}{j+\alpha+1}+D=\frac{3j+2\alpha+2}{j+\alpha+1}+D.
\end{equation*}
In fact, numerical experiments indicate that the sequence $\frac{2M\big(C_{2n}^{(\alpha)}(1)\big)^{(j)}\kappa_{2(n-1)}^{(j,0)}(1,1)}{1+2M\kappa_{2(n-1)}^{(j,j)}(1,1)}$ is decreasing, so $D$ cannot be removed.

\item Case $-1<\alpha<-1/2. $ For our purpose it is easier to take into account (\ref{RCF}). In this way, we get
\begin{eqnarray*} \label{cotas-sb}
\max_{x\in[-1,1]}\left|Q_n^{(\alpha,M,j)}(x)\right|
&\leq&\sum_{i=0}^{j+1} \max_{x\in[-1,1]}\left|\gamma_{n,i}(1-x^2)^i\big(C_{n-i}^{(\alpha+i)}(x)\big)^{(i)}\right| \nonumber\\
&\leq& (j+2)\max_{i\in \{0,\ldots, j+1\}} \max_{x\in[-1,1]}\left|\gamma_{n,i}(1-x^2)^i\big(C_{n-i}^{(\alpha+i)}(x)\big)^{(i)}\right|.
\end{eqnarray*}
We are going to compute $\max_{x\in [-1,1]}\left|\gamma_{n,i}(1-x^2)^i\big(C_{n-i}^{(\alpha+i)}(x)\big)^{(i)}\right|.$ First, we observe that using (\ref{deri}) we get
\begin{eqnarray*}\label{RCF2}
(1-x^2)^i\big(C_{n-i}^{(\alpha+i)}(x)\big)^{(i)}&=&(1-x^2)^i\rho_{n,i}C_{n-2i}^{(\alpha+2i)}(x) \nonumber\\ \nonumber
&=&(1-x^2)^i\rho_{n,i}\frac{\Gamma(n-2i+1)\Gamma(\alpha+2i+1)}{\Gamma(n+\alpha+1)}P_{n-2i}^{(\alpha+2i,\alpha+2i)}(x),
\end{eqnarray*}
where \begin{equation} \label{ros}\rho_{n,i}=\frac{(-1)^i(n+2\alpha+i+1)_i(-n+i)_i}{2^i(\alpha+i+1)_i}\approx \frac{1}{2^i(\alpha+i+1)_i}  n^{2i}.\end{equation}

Using (\ref{ros}), (\ref{bound-jacobi}) and Proposition \ref{conv gamma i},  we get for $n$ large enough,

\begin{eqnarray*}
&&\max_{x\in[-1,1]}\left|\gamma_{n,i} (1-x^2)^i\rho_{n,i}\frac{\Gamma(n-2i+1)\Gamma(\alpha+2i+1)}{\Gamma(n+\alpha+1)}P_{n-2i}^{(\alpha+2i,\alpha+2i)}(x)\right|\\
&&\le C_2\ n^{2i}n^{-\alpha-2i}n^{-1/2}=C_2\ n^{-\alpha-1/2},
\end{eqnarray*}
which proves the result for this case.
\end{itemize}

\end{proof}


\section{Asymptotics behavior of the eigenvalues of Gegenbauer-Sobolev orthogonal polynomials}

In \cite{Bavinck1999} the authors claim that there exits a linear differential operator of the form $\mathbf{T}=\mathbf{L}+M\mathbf{A}$ for discrete Sobolev orthogonal polynomials with respect to an inner product such as
\begin{equation*}
(f,g)=\int_I f(x)g(x)d\mu+M\big[f^{(j)}(-c)g^{(j)}(-c)+f^{(j)}(c)g^{(j)}(c)\big], \quad c>0,
\end{equation*}
where $\mu$ is a finite symmetric Borel measure supported on the interval $I$.  $\mathbf{L}$ is the linear differential operator associated with the standard polynomials orthogonal with respect to $\mu.$ This operator $\mathbf{L}+M\mathbf{A}$ can have infinite order. Obviously, the inner product (\ref{pro}) here considered lies in this framework.

In addition, the authors give expressions for the eigenvalues associated with  $\mathbf{L}+M\mathbf{A}.$ Then, if we particularize this for the Gegenbauer--Sobolev orthogonal polynomials, we have
$$(\mathbf{L}+M\mathbf{A})Q_n^{(\alpha,M,j)}(x)=\widetilde{\lambda}_nQ_n^{(\alpha,M,j)}(x).$$
We are looking for the asymptotic behavior of $\widetilde{\lambda}_n$ which is the key to establish one of our main goals in this work.

Following \cite{Bavinck1999}, we get that $\widetilde{\lambda}_n=\lambda_n+M\mu_n$, where the numbers $\{\mu_m\}_{m=0}^{j+1}$ can be chosen arbitrarily and $\{\mu_m\}_{m=j+2}^\infty$ and the operator $\mathbf{A}$ are uniquely determined once  the choice of these arbitrary numbers has been done. In fact, they establish
\begin{eqnarray*}
\mu_{j+2t}&=&\mu_j+\sum_{i=1}^t(\lambda_{j+2i}-\lambda_{j+2i-2})q_{j+2i,j+2i},\\
\mu_{j+2t+1}&=&\mu_{j+1}+\sum_{i=1}^t(\lambda_{j+2i+1}-\lambda_{j+2i-1})q_{j+2i+1,j+2i+1},
\end{eqnarray*}
where $t$ is a positive integer and
 $$ q_{n,n}=K_{n-1}^{(j,j)}(1,1)+(-1)^{n+j}K_{n-1}^{(j,j)}(1,-1).$$
 Since $\{\mu_m\}_{m=0}^{j+1}$ can be chosen arbitrarily, for simplicity we take $\mu_0=\dots=\mu_{j+1}=0.$ Using (\ref{EiV}), we obtain

\begin{eqnarray} \label{mu-bav} \mu_{j+2t}&=&2\sum_{i=1}^t(2j+4i+2\alpha-1)q_{j+2i,j+2i},\\
\mu_{j+2t+1}&=&2\sum_{i=1}^t(2j+4i+2\alpha+1)q_{j+2i+1,j+2i+1}. \label{mu-bav-2}
\end{eqnarray}
 We are going to establish the asymptotic behavior of the sequence $\{\mu_n\}_{n}$ given by (\ref{mu-bav})--(\ref{mu-bav-2}) when $n\to \infty.$ First, we need a technical result.
\begin{prop}
We have,
\begin{equation}\label{qnn even}
q_{s+2i,s+2i}= \left\{
                  \begin{array}{ll}
                    2\kappa_{s+2i-2}^{(j,j)}(1,1), & \hbox{if $s$ is even;} \\ \\
                    2\widetilde{\kappa}_{s+2i-3}^{(j,j)}(1,1), & \hbox{if $s$ is odd,}
                  \end{array}
                \right.
\end{equation}
where $\kappa_{2m}^{(j,k)}(x,y)$ and $\widetilde{\kappa}_{2m}^{(j,k)}(x,y)$ are given in (\ref{evenkernels}) and (\ref{oddkernels}), respectively.
\end{prop}
\begin{proof} We use the definition of $q_{s,s}.$
\begin{eqnarray*}
q_{s+2i,s+2i}&=&K_{s+2i-1}^{(j,j)}(1,1)+(-1)^{s+j}K_{s+2i-1}^{(j,j)}(1,-1)\\
&=&\sum_{m=0}^{s+2i-1}\frac{\left(\big(C_m^{(\alpha)}(1)\big)^{(j)}\right)^2}{||C_m^{(\alpha)}||^2_{\alpha}}
+(-1)^{s+j}\sum_{m=0}^{s+2i-1}\frac{\big(C_m^{(\alpha)}(1)\big)^{(j)}\big(C_m^{(\alpha)}(-1)\big)^{(j)}}{||C_m^{(\alpha)}||^2_{\alpha}}\\
&=&\sum_{m=0}^{s+2i-1}\frac{\left(\big(C_m^{(\alpha)}(1)\big)^{(j)}\right)^2\left(1+(-1)^{s+m+2j}\right)}{||C_m^{(\alpha)}||^2_{\alpha}}\\
\end{eqnarray*}
Then, if $s$ is even we get
$$
q_{s+2i,s+2i}=2\sum_{m=0, \, m \, even}^{s+2i-1}\frac{\left(\big(C_m^{(\alpha)}(1)\big)^{(j)}\right)^2}{||C_m^{(\alpha)}||^2_{\alpha}}= 2\kappa_{s+2i-2}^{(j,j)}(1,1).
$$
The odd case is established in the same way.
\end{proof}

\begin{prop} \label{ABmu}
It holds
\begin{equation*}
\lim_{n\to+\infty}\frac{\mu_{2n}}{n^{4j+2\alpha+4}}=\lim_{n\to+\infty}\frac{\mu_{2n+1}}{n^{4j+2\alpha+4}}
=\frac{2^{2j+3}}{(2j+\alpha+2)(2j+\alpha+1)\Gamma^2(\alpha+j+1)}.
\end{equation*}
\end{prop}
\begin{proof} For $n$ large enough we can write  $2n=2m+j,$ and so $j$ is even.  To establish this result we are going to use the Stolz's criterium and formulae (\ref{kappa(k,s)(1,1)}) and (\ref{qnn even}).
\begin{eqnarray*}
&&\lim_{n\to+\infty}\frac{\mu_{2n}}{n^{4j+2\alpha+4}}=\lim_{m\to+\infty}\frac{\mu_{j+2m}}{m^{4j+2\alpha+4}}\\
&&=\lim_{m\to+\infty}\frac{2\sum_{i=1}^{m}(2j+4i+2\alpha-1)q_{j+2i,j+2i}-2\sum_{i=1}^{m-1}(2j+4i+2\alpha-1)q_{j+2i,  j+2i}}{m^{4j+2\alpha+4}-(m-1)^{4j+2\alpha+4}}\\
&&=\frac{1}{(2j+\alpha+2)}\lim_{m\to+\infty}\frac{(2j+4m+2\alpha-1)q_{j+2m, j+2m}}{m^{4j+2\alpha+3}}\\
&&=\frac{1}{2j+\alpha+2}\lim_{m\to+\infty}\frac{(2j+4m+2\alpha+1)}{m}
\frac{2\kappa_{2(m-1+j/2)}^{(j,j)}(1,1)}{m^{4j+2\alpha+2}}\\
&&=\frac{2^{2j+3}}{(2j+\alpha+2)(2j+\alpha+1)\Gamma^2(\alpha+j+1)}.
\end{eqnarray*}
Analogously, for $n$ large enough $2n+1=2m+j+1, $ so $j+1$ is odd. Then, to prove the other limit we can use (\ref{qnn even}) with $s=j+1.$
\end{proof}

Finally, we are ready to establish the asymptotic behavior of the eigenvalues $\widetilde{\lambda}_n. $

\begin{prop} \label{ABeiv}
Let $\widetilde{\lambda}_n$ be the eigenvalues associated with the linear differential operator  $\mathbf{T}=\mathbf{L}+M\mathbf{A}.$ Then,
\begin{equation*}
\lim_{n\to+\infty}\frac{\widetilde{\lambda}_n}{n^{4j+2\alpha+4}}=\frac{M}{2^{2j+2\alpha+1}(2j+\alpha+2)(2j+\alpha+1)\Gamma^2(\alpha+j+1)}.
\end{equation*}
\end{prop}
\begin{proof} Applying (\ref{EiV}) and Proposition \ref{ABmu}
\begin{eqnarray*}
\lim_{n\to+\infty}\frac{\widetilde{\lambda}_n}{n^{4j+2\alpha+4}}&=&\lim_{n\to+\infty}\frac{\lambda_n+M\mu_n}{n^{4j+2\alpha+4}}=
\lim_{n\to+\infty}\frac{M\mu_n}{(n/2)^{4j+2\alpha+4}\ 2^{4j+2\alpha+4}}\\
&=&\frac{M}{(2j+\alpha+2)(2j+\alpha+1)\Gamma^2(\alpha+j+1)2^{2j+2\alpha+1}}.
\end{eqnarray*}
\end{proof}

To conclude this section, we establish one of the main goals of this paper.

\begin{theo}\label{main1}
Let $\widetilde{Q}_n^{(\alpha,M,j)}(x)$ be orthonormal polynomials with respect to (\ref{pro}), and $\widetilde{\lambda}_n$ the eigenvalues associated with the linear differential operator  $\mathbf{T}=\mathbf{L}+M\mathbf{A}.$ Then,
\begin{itemize}
  \item if $\alpha\geq-1/2,$
  $$\lim_{n\to+\infty}\frac{\log( \max_{x\in [-1,1]}|\widetilde{Q}_n^{(\alpha,M,j)}(x)|)}{\log(\widetilde{\lambda}_n)}=\frac{\alpha+1/2}{4j+2\alpha+4},$$
  \item if $-1<\alpha<-1/2,$
  $$\lim_{n\to+\infty}\frac{\log( \max_{x\in [-1,1]}|\widetilde{Q}_n^{(\alpha,M,j)}(x)|)}{\log(\widetilde{\lambda}_n)}=0.$$
\end{itemize}
\end{theo}
\begin{proof} We are going to use Theorem \ref{cota Bn}, Proposition \ref{p-ABnorm2n}, Proposition \ref{ABeiv}, and (\ref{asNORM}).
\begin{itemize}
  \item Case $\alpha\geq-1/2.$ We have
\begin{eqnarray*}
&&\frac{\log(\max_{x\in[-1,1]}|\widetilde{Q}_n^{(\alpha,M,j)}(x)|)}{\log(\widetilde{\lambda}_n)}= \frac{\log(\max_{x\in[-1,1]}|Q_n^{(\alpha,M,j)}(x)|/||Q_n^{(\alpha,M,j)}||_S)}{\log(\widetilde{\lambda}_n)}\\
&&=\frac{\log\left(\max_{x\in[-1,1]}|Q_n^{(\alpha,M,j)}(x)|/\big(||Q_n^{(\alpha,M,j)}||_S\, n^{\alpha+1/2}n^{-\alpha-1/2}\big)\right)}
{\log\left(\frac{\widetilde{\lambda}_n}{n^{4j+2\alpha+4}}n^{4j+2\alpha+4}\right)}\\
&&=\frac{\log(\max_{x\in[-1,1]}|Q_n^{(\alpha,M,j)}(x)|)-\log\left(||Q_n^{(\alpha,M,j)}||_S\,  n^{\alpha+1/2}\right)-(-\alpha-1/2)\log(n)}
{(4j+2\alpha+4)\log(n) +\log\left(\frac{\widetilde{\lambda}_n}{n^{4j+2\alpha+4}}\right)}
\end{eqnarray*}
Taking limits, we obtain
 $$\lim_{n\to+\infty}\frac{\log( \max_{x\in [-1,1]}|\widetilde{Q}_n^{(\alpha,M,j)}(x)|)}{\log(\widetilde{\lambda}_n)}=\frac{\alpha+1/2}{4j+2\alpha+4}.$$

  \item Case $-1<\alpha<-1/2.$ If we proceed like in the above case,  then we get
  \begin{eqnarray*}
&&\frac{\log(\max_{x\in[-1,1]}|\widetilde{Q}_n^{(\alpha,M,j)}(x)|)}{\log(\widetilde{\lambda}_n)}
\leq \frac{\log(Fn^{-\alpha-1/2}/||Q_n^{(\alpha,M,j)}||_S)}{\log(\widetilde{\lambda}_n)}\\
&&=\frac{\log(F)-\log\left(n^{\alpha+1/2}||Q_n^{(\alpha,M,j)}||_S \right)}
{(4j+2\alpha+4)\log(n) +\log\left(\frac{\widetilde{\lambda}_n}{n^{4j+2\alpha+4}}\right)} \to 0, \quad n\to \infty,
\end{eqnarray*}
which proves the result.
\end{itemize}
\end{proof}

\section{Mehler--Heine asymptotics for Gegenbauer--Sobolev orthogonal polynomials}

 Mehler-Heine formulae  are very relevant because they describe in  detail the asymptotic behavior around point $x=1$ where we have located the perturbation (using the symmetry of these polynomials we also have the information around the point $x=-1$).  This type of asymptotics has been considered in several frameworks. In the context of Sobolev orthogonality  there is a wide  literature, we can cite the surveys \cite{mmb} and \cite{paco-xu}, and the references therein. Even more recently and conceptually closer to the inner product (\ref{pro}) we can point out \cite{mavjs,SGV16,dls2015} among others.

To establish Mehler--Heine formula  for the discrete Gegenbauer--Sobolev orthogonal polynomials considered in this work, we need the  corresponding formula for classical Jacobi orthogonal polynomials. For $\alpha,\beta$ real numbers and $s$ an integer number, it holds (see \cite[Th. 8.1.1]{sz}):
\begin{equation}\label{mhj}
\lim_{n\to \infty}n^{-\alpha}P_n^{(\alpha,\beta)}\left(\cos\left(\frac{x}{n+s}\right)\right)=\lim_{n\to\infty}\frac{1}{n^{\alpha}}P_n^{(\alpha,\beta)}
\left(1-\frac{x^2}{2(n+s)^2}\right)=(x/2)^{-\alpha}J_{\alpha}(x),
\end{equation}
uniformly on compact subsets of $\mathbb{C}, $ where $J_{\alpha}(x)$ denotes the  Bessel function of the first kind, i.e.,
\begin{equation*}
J_{\alpha}(x)=\sum_{k=0}^{\infty} \frac{(-1)^k}{k!\Gamma(k+\alpha +1)} \left( \frac{x}{2} \right)^{2k+\alpha}.
\end{equation*}
The integer number $s$ will play an important role in the proof of Theorem \ref{amh GS}. The original statement of Mehler--Heine formula for classical Jacobi polynomials was made with $s=0, $ but it can be extended for every integer number $s$ as it was established in the proof of Corollary 1 in \cite{alf-mb-pen-rez}. In that paper it was proved a more general result: if $(f_n)_n$ is a sequence of  holomorphic functions on  $\mathbb{C}$ and $(b_n)_n$  is a sequence of complex numbers satisfying $\lim_{n \to \infty} \frac{b_n}{b_{n+s}}=1$ for every integer number $s $  such that  $\left(f_n(z/b_n)\right)_n$ converges to a function $f$ uniformly on compact subsets of $\mathbb{C}$,
 then
$$\lim_{n \to \infty} f_n\left(\frac{z}{b_{n+s}}\right) = f(z)$$
uniformly on compact subsets of $\mathbb{C}$ for every integer $s.$

Thus, we can claim:

\begin{theo}\label{amh GS}
For the sequence  $\{Q_n^{(\alpha,M,j)}\}_{n\geq 0}$ the following Mehler--Heine formula holds
\begin{equation}\label{MH GS}
\lim_{n\to \infty}Q_n^{(\alpha,M,j)}\left(\cos\left(\frac{x}{n}\right)\right)=\lim_{n\to+\infty}Q_n^{(\alpha,M,j)}\left(1-\frac{x^2}{2n^2}\right)
=\varphi_{\alpha,j}(x),
\end{equation}
uniformly on compact subsets of $\mathbb{C}$, where \begin{equation} \label{z-f-l} \varphi_{\alpha,j}(x)=\sum_{i=0}^{j+1}2^i\gamma_{i}\Gamma(\alpha+i+1)(x/2)^{-\alpha}J_{\alpha+2i}(x),
 \end{equation} with the coefficients $\gamma_i$ given in (\ref{gamma_i}).
\end{theo}

\begin{proof} Scaling adequately in (\ref{RCF}) and using (\ref{deri}), (\ref{RGJ}) and (\ref{ros}), we get
\begin{eqnarray*}
&&\lim_{n\to+\infty}Q_n^{(\alpha,M,j)}\left(1-\frac{x^2}{2n^2}\right)=\\
&&\lim_{n\to+\infty}\sum_{i=0}^{j+1}\gamma_{n,i}\rho_{n,i}\frac{x^{2i}}{n^{2i}}\left(1-\frac{x^2}{4n^2}\right)^i
\frac{\Gamma(n-2i+1)\Gamma(\alpha+2i+1)}{\Gamma(n+\alpha+1)}P_{n-2i}^{(\alpha+2i, \alpha+2i)}\left(1-\frac{x^2}{2n^2}\right)=\\
&&\lim_{n\to+\infty}\sum_{i=0}^{j+1}\gamma_{n,i}x^{2i}\frac{\rho_{n,i}}{n^{2i}}\left(1-\frac{x^2}{4n^2}\right)^i
\frac{n^{\alpha+2i} \Gamma(n-2i+1)\Gamma(\alpha+2i+1) }{\Gamma(n+\alpha+1)}\frac{P_{n-2i}^{(\alpha+2i, \alpha+2i)}\left(1-\frac{x^2}{2n^2}\right)}{n^{\alpha+2i}}.
\end{eqnarray*}
It only remains to apply the asymptotic behaviors given by (\ref{stirling}), (\ref{ros}) and (\ref{mhj}) to obtain the result.
\end{proof}

Next, we are going to pay attention  to the zeros of the polynomials $Q_n^{(\alpha,M,j)}.$ When $j=0$ the inner product (\ref{pro}) is standard, i.e, it is related to the measure $\mu$ given by $d\mu=(1-x)^{\alpha}(1+x)^{\beta}dx+M(\delta (x+1)+\delta(x-1))$ where $\delta(x)$ is the Dirac's delta function. Thus, all the zeros of $Q_n^{(\alpha,M,j)}$ are real and they are within $(-1,1).$

However, when $j>0$ the situation changes. It was proved by  H. G. Meijer in \cite[Th. 4.1]{mei} (see also \cite[Lemma 2]{allore}) in a more general framework that the polynomial $Q_{n}^{(\alpha,M,j)}(x),$ $n\geq 1,$   has $n$ real and simple zeros and at most two of them are located outside  $(-1,1).$ However, on the one hand, using Proposition \ref{ASeven} we have that $\displaystyle \lim_{n\to+\infty}Q_n^{(\alpha,M,j)}(1)=\frac{-j}{j+\alpha+1}<0,$ but on  the other hand, the leading coefficient of $Q_{n}^{(\alpha,M,j)}(x)$ is $k_n(\alpha)>0$ given by (\ref{lc}), so we have  $\displaystyle\lim_{x\to+\infty}Q_n^{(\alpha,M,j)}(x)=+\infty.$ Thus, we deduce that there exists a zero within $(1,+\infty)$ and, by the symmetry of the polynomials, another one within $(-\infty,-1). $ Therefore, we can summarize it in the following result.
\begin{prop}\label{zeros-q}
If $j>0$ the polynomial $Q_{n}^{(\alpha,M,j)}(x),$ $n\geq 1,$   has $n$ real and simple zeros and exactly two of them are located outside  $(-1,1).$ If $j=0, $ then all the zeros are within $(-1,1).$
\end{prop}

Now, we conclude giving the asymptotic behavior of the zeros. It is a well--known consequence of  Theorem \ref{amh GS}. It is only necessary to apply Hurwitz's theorem (see \cite[Th. 1.91.3]{sz}) to (\ref{MH GS}). We denote by $[a]$ the integer part of $a.$ Thus, we have

\begin{prop}\label{zeros-ab} For $j>0,$ we denote by $s_{n,i}, $ $i=1,\ldots, [n/2]-1, $ the $ [n/2]-1$ positive zeros of $Q_{n}^{(\alpha,M,j)}$ within $(0,1)$ in a decreasing order, i.e., $s_{n, [n/2]-1}< s_{n, [n/2]-2}<\cdots< s_{n,1}.$ Then,
$$\lim_{n\to \infty} n\arccos(s_{n,i})=y_i, \quad i=1,\ldots, [n/2]-1,
$$
where $0<y_1<\cdots <y_{[n/2]-1}$ denote the first $[n/2]-1$ positive real zeros of the function $\varphi_{\alpha,j}$ given in (\ref{z-f-l}).
\end{prop}

For $j=0$  the inner product (\ref{pro}) appears in \cite{bra-cas-mb} as a very particular case in a context which involves continuous Sobolev polynomials. Notice that in this situation using Corollary \ref{limit gamma i} we deduce that the limit function (\ref{z-f-l}) is
$$\varphi_{\alpha,0}(x)= -\Gamma(\alpha+1) (x/2)^{-\alpha}J_{\alpha+2}(x).$$
The above result was obtained in \cite[Proposition 2]{bra-cas-mb} but in that paper the inner product considered was
\begin{equation*}
(f,g)_S:=\int_{-1}^1f(x)g(x)(1-x^2)^{\alpha-1/2}dx+M\left(f(-1)g(-1)+f(1)g(1)\right),
\end{equation*}
where the corresponding Sobolev orthogonal polynomials were  monic. Then, to compare both results we must take into account these facts. Anyway, the zeros of $Q_{n}^{(\alpha,M,0)}$  behave asymptotically like the zeros of $x^{-\alpha}J_{\alpha+2}(x)$ (obviously in \cite[Proposition 2]{bra-cas-mb} $\alpha$ must be changed by $\alpha+1/2$). Thus, following the above notation,  in the case $j=0$ we get
$$ \lim_{n\to \infty}s_{n,1}=1, \quad \lim_{n\to \infty} n\arccos(s_{n,i})=j_i^{(\alpha+2)}, \quad i=2,\ldots, [n/2],
$$
where $0<j_1^{(\alpha+2)}<\cdots< j_{[n/2]}^{(\alpha+2)}$ denote the first $[n/2]$ positive zeros of the Bessel function of the first kind $J_{\alpha+2}.$

\bigskip

\noindent \textbf{Acknowledgments:} The authors JFMM and JJMB are partially supported  by Research Group FQM-0229 (belonging to Campus of International Excellence CEIMAR). The author JFMM is funded by a grant of Plan Propio de la Universidad de Almer\'{\i}a. The author JJMB is partially supported by Ministerio de Econom\'{i}a y Competitividad of Spain and European Regional Development Fund, grant  MTM2014-53963-P, and Junta de Andaluc\'{i}a (excellence grant P11-FQM-7276).
The author JFMM is grateful to Lance L. Littlejohn and the Department of Mathematics of Baylor University for an invitation to a research stay at Baylor University. This work was partially carried out during a research stay of the author JFMM at this university.


\end{document}